\documentclass[12pt]{article}
\usepackage{amsmath,amssymb,amsthm,amsfonts,amscd}
\usepackage{hyperref}
\usepackage[numeric]{amsrefs}
\usepackage{geometry}
\usepackage{rotating} %
\input colordvi
\usepackage{graphicx}
\let\oldbibliography\thebibliography
\renewcommand{\thebibliography}[1]{%
	\oldbibliography{#1}%
	\setlength{\itemsep}{0pt}%
}

\newtheorem{thm}[equation]{Theorem}
\newtheorem{cor}[equation]{Corollary}
\newtheorem{lem}[equation]{Lemma}

\newcommand{\thmref}[1]{Theorem~\ref{#1}}

\numberwithin{equation}{section}

\newcommand\e{\varepsilon}
\renewcommand\l{\lambda}
\renewcommand\L{\Lambda}

\newcommand\f{\frac}
\newcommand\smallf[2]{{\textstyle{\frac{#1}{#2}}}}

\newcommand{\Z}{{\mathbb{Z}}}
\newcommand{\R}{{\mathbb{R}}}

\newcommand{\C}{{\mathbb{C}}}

\DeclareMathOperator{\sech}{sech}

\renewcommand\i{^{-1}}
\renewcommand\({\left(}
\renewcommand\){\right)}

\newcommand\srel[2]{\begin{smallmatrix} {#1} \\ {#2} \end{smallmatrix}}

\newcommand{\gobble}[1]{}
  \newcommand{\rangeref}[2]{%
    \ref{#1}--\afterassignment\gobble\fam 0\ref{#2}%
  }

\makeatletter
\def\imod#1{\allowbreak\mkern5mu({\operator@font mod}\,#1)}
\makeatother

\title{Kissing numbers and transference theorems from generalized tail bounds}

\author{Stephen D.~Miller\thanks{Rutgers University. Supported by NSF grant CNS-1526333.}
\\
Noah Stephens-Davidowitz\thanks{Massachusetts Institute of Technology.}}

\begin{document}

\maketitle

\vspace{-1cm}

\begin{abstract}
We generalize Banaszczyk's  seminal tail bound for the Gaussian mass of a lattice to a wide class of test functions.  From this we obtain quite general transference bounds, as well as bounds on the number of lattice points contained in certain bodies.   As  applications, we bound the lattice kissing number in $\ell_p$ norms by $e^{(n+ o(n))/p}$ for $0 < p \leq 2$, and also  give a  proof of a new transference bound in the $\ell_1$ norm.
\end{abstract}

\section{Introduction}\label{sec:Introduction}

A lattice $\L \subset \R^n$ is the set of integral linear combinations of some  basis   $\{b_1, \ldots, b_n\}$ of $\R^n$.
The dual lattice
\[
\L^*= \{x\in \R^n \ : \ x\cdot \lambda\in \Z,\ \forall \lambda\in \L\}
\]
is the set of vectors that have integer inner product with all lattice vectors, and is itself a lattice satisfying $(\L^*)^* = \L$.
A ubiquitous classical tool for studying lattices (with many applications in fields as diverse as number theory and computer science) is the {\it Gaussian mass}
\[
\sum_{\l \in \L} e^{-\pi  \|\l + v\|_2^2}
\; ,
\]
for $v \in \R^n$, where $\|x\|_2 := (x_1^2 + x_2^2 + \cdots+ x_n^2)^{1/2}$ is the Euclidean norm of $x = (x_1,\ldots, x_n) \in \R^n$. (See, for example, \cite{Jacobi1828,Riemann, MO,Banaszczyk, Cai,BianePY01,SS06,Mumford,MR,katz,oded05,RS17,NSDthesis}.)  The case of $v=0$ specializes to the usual $\theta$-function of the lattice $\Lambda$.

 Banaszczyk~\cite{Banaszczyk} proved an important \emph{tail bound} on the Gaussian mass of lattice points outside of a ball,
\begin{equation}
	\label{eq:banaszczyk}
	\sum_{\stackrel{\l \, \in\, \L}{\|\l + v\|_2\, \geq \,r}} e^{-\pi  \|\l + v\|_2^2}  \ \ \leq \ \  (2\pi  e n\i r^2)^{n/2} \, e^{-\pi  r^2}  \sum_{\l \, \in\, \L} e^{-\pi  \|\l\|_2^2}
\end{equation}
for any $r \geq \sqrt{\frac{n}{2\pi}}$ and {\it any} lattice $\Lambda$. He then used this bound to prove nearly optimal \emph{transference theorems}, which  relate the   geometry of $\L$ to that of $\L^*$ (see Section~\ref{sec:banaszczyk_recovered}).
 Both the tail  bound and the transference theorems have since found many additional applications in the study of the geometry of lattices (e.g.,~\cite{Banaszczykpolar,Cai}), algorithms for computational problems over lattices (e.g.,~\cite{Klein00,LiuLM06,NguyenVidick08,ADRS15}), the complexity of such problems (e.g.,~\cite{AharonovR04,MR,oded05}), and lattice-based cryptography (e.g.,~\cite{GPV08,Gentry09,Peikert09}), among other fields.

\subsection{Tail and transference bounds beyond Gaussians}

Given the  importance of \eqref{eq:banaszczyk}, we find it natural to ask for which test functions $f:\R^n \to \R_{\ge  0}$ and  subsets $K \subset \R^n$ one can obtain estimates for sums of the form
\begin{equation}\label{unnumberedequation}
	\sum_{\stackrel{\l \in \L}{\lambda + v \notin K}} f(\l+ v)
	\; .
\end{equation}
 For example, our application  in \thmref{thm:ell1_transference} uses the test function
 $f(x_1,\ldots, x_n) = \prod_i (1+2\cosh(2\pi x_i/\sqrt{3}))^{-1}$, while our application in Theorem~\ref{thm:handshake} uses $f(x) = e^{-\|x\|_p^p}$ for $0<p\le 2$, where $\|x\|_p = \|(x_1,\ldots, x_n)\|_p := (|x_1|^p + \cdots + |x_n|^p)^{1/p}$. To that end, we generalize Banaszczyk's elegant Fourier-analytic proof of~\eqref{eq:banaszczyk} into a more flexible framework (see Section~\ref{sec:Poissonsum}).
For example, we prove the tail bound
\begin{flalign}
\label{eq:tail_bound_intro}
\text{({\textbf{\thmref{mainthm}, Part~\ref{item:main_two}}}:)}   	&&\sum_{\srel{\l\,\in\,\L}{\l + v \notin K}} f(\l+v) \ \ \leq \ \ \nu_f(K) \sum_{\l\,\in\,\L}f(\l) &&
\end{flalign}
for any subset $K \subset \R^n$, where
\begin{equation}
	\label{eq:nu_def_intro}
	\nu_f(K) \ \ := \ \ \inf_{0 < u \le  1} \, \sup_{x \, \notin \, K}   \frac{f(x)}{u^n f(ux)}
	\,.
\end{equation}
 (See \thmref{mainthm} for precise conditions on the function $f$.)
 We also show in Corollary~\ref{cor:main_theorem_norms} that the bound \eqref{eq:tail_bound_intro} takes a particularly nice form  for compact sets  $K \subset \R^n$ which are   starlike with respect to the origin,\footnote{In terms of \eqref{normK}, this means for each $r>0$ we have $x\notin rK\Longleftrightarrow \|x\|_K>r$.} and for functions (such as Gaussians) that satisfy a certain concavity condition (see~\eqref{ratiomonot}) and which depend only on the ``norm''
\begin{equation}\label{normK}
\|x\|_K  \ \ :=  \ \ \min \{r \geq 0 \ : \ x \in rK \}\,.
\end{equation}

Next, following  Banaszczyk's approach \cite{Banaszczyk},  we use~\eqref{eq:tail_bound_intro} to show a general transference bound in Theorem~\ref{thm:generic_transference}, which relates the geometry of $\L$ and $\L^*$.
To that end, for any starlike compact set $K \subset \R^n$ with the origin in its interior, and any lattice $\L \subset \R^n$, we define
\begin{align}
\label{eq:sigma_K_def}
\sigma_{K}(\L) & \ \ :=  \ \ \min_{\l \in \L_{\neq 0}} \|\l \|_K \; , \\
\text{ and} \ \ \
\label{eq:rho_K_def}
\rho_{K}(\L) & \ \ :=  \ \ \max_{v \in \R^n} \min_{\l \in \L} \|\l-v\|_K
\;.
\end{align}
 I.e., $\sigma_K(\L)$ is the length of the shortest non-zero vector and $\rho_K(\L)$ is the covering radius in the $\|\cdot \|_K$ ``norm.'' We show that
\begin{flalign}
\label{eq:generic_transference_intro}
\text{{(\textbf{\thmref{thm:generic_transference}}:)}}
&&
\sigma_{K}(\L) \,  \rho_{K'}(\L^*)   \ \   \leq  \ \   1 &&
\end{flalign}
for any suitable sets $K, K' \subset \R^n$ such that $2\nu_f(K) + \nu_{\widehat{f}}(K') < 1$ for some Fourier transform pair of functions $f,\widehat{f}$ satisfying certain analytic conditions.
In particular, taking $f$ to be the Gaussian and $K = K'$ to be a Euclidean ball of a certain radius immediately recovers Banaszczyk's Euclidean transference bound:
\begin{equation}
\label{eq:ell2_transference_intro}
	\big(\min_{\lambda\,\in\,\L_{\neq 0}} \|\l\|_2\big) \big( \max_{v\,\in\,\R^n}  \min_{\lambda\,\in\,\L^*}  \|\l-v\|_2\big) \ \  \leq \ \  \frac{n}{2\pi} + \frac{3\sqrt{n}}{\pi}
	\; .
\end{equation}
(Banaszczyk actually stated a slightly weaker result, but he noted that his proof actually yields something like~\eqref{eq:ell2_transference_intro}. See Section~\ref{sec:banaszczyk_recovered}.)

\subsection{Applications with carefully chosen test functions }

We then derive applications of~\eqref{eq:tail_bound_intro} and~\eqref{eq:generic_transference_intro} improving on previous trivial bounds, using functions $f$ whose analytic properties are tailored to the geometry at hand.

We use the function $f(x_1,\ldots, x_n) = \prod_i (1+2\cosh(2\pi x_i/\sqrt{3}))^{-1}$ to prove a transference bound in the $\ell_1$ norm,
\begin{flalign}\label{l1transfintro}
\text{(\textbf{Theorem~\ref{thm:ell1_transference}}:)} \!\!
&&  (\min_{\lambda \in \L_{\neq 0}} \! \! \|\l\|_1)  ( \max_{v \in \R^n} \min_{\lambda \in \L^*} \!  \|\l-v\|_1)
  \ \  <  \ \  c_1 n^2  (1 + 2\pi \sqrt{\smallf{3}{n}} )^2  ,
\end{flalign}
with $c_1 \approx 0.154264$. In~\cite{Banaszczykpolar}, Banaszczyk proved more general transference bounds that apply for arbitrary $\ell_p$ norms for $1 \leq p \leq \infty$, but with  unspecified constants. Previously the best known bound was the (just slightly weaker) trivial estimate with
 $c_1 = \f{1}{2\pi} + o(1) \approx 0.159155 + o(1)$, which follows immediately from~\eqref{eq:ell2_transference_intro} together with the Cauchy-Schwarz inequality $\|x\|_1 \leq \sqrt{n}  \|x\|_2$.

Finally, in Theorem~\ref{thm:handshake}, we use the functions
\begin{equation}
\label{eq:ell_p_def_intro}
f(x) \ \  = \ \  e^{-\|x\|_p^p}
\end{equation}
to prove bounds on the lattice kissing number (also known as the lattice Hadwiger number) of the $\ell_p$ balls with $0 < p \leq 2$.  Namely, we show that for such $p$
\begin{flalign}\label{lpkissing}
\text{(\textbf{Theorem~\ref{thm:handshake}}:)}
&&\#\{\lambda\,\in\,\Lambda \,:\, \|\l\|_p = \sigma_p(\L)\}
\ \  \leq \  \  O(\smallf np e^{n/p})\,, &&
\end{flalign}
where $\sigma_p(\lambda)=\min_{\lambda\in\L_{\neq 0}}\|\l\|_p$.
To the authors' knowledge, these are the best bounds presently known for $1/\log 2 < p < 2$ and for $0 < p \le 1$ --- in particular, including the  case of $p = 1$.  (See the discussion above Theorem~\ref{thm:handshake}.) Theorem~\ref{thm:handshake} actually gives a more general result:~a bound on the number of non-zero vectors whose $\ell_p$ norm is within some factor $u \geq 1$ of the minimal value.

It is a pleasure to thank our colleagues  Divesh Aggarwal,  Tamar Lichter, Assaf Naor, Chris Peikert, Oded Regev, Konrad J. Swanepoel, and Ramarathnam Venkatesan for their helpful discussions and comments. We also thank the anonymous reviewers for their very helpful comments.

\section{Poisson summation and tail bounds}\label{sec:Poissonsum}

 We begin with the following version of the Poisson summation formula:
\begin{equation}\label{PSF}
\sum_{\l\in\L}f(\smallf{\l+v}{t}) \ \ = \ \ \f{t^n}{|\L|}\sum_{\l\in\L^*}\widehat{f}(t\l)\,e(t\l\cdot v)\,, \ \ \ t\,>\,0 \ \ \text{and} \ \ v\,\in\,\R^n\,,
\end{equation}
where
$e(y) := e^{2\pi i y}$ and $\widehat{f}(x) :=\int_{\R^n}f(r)e(-r\cdot x) {\rm d}r$ is the Fourier transform of $f:\R^n\rightarrow\C$.  Here in order to justify applying this formula we assume that
\begin{equation}\label{PSFconditions}\aligned
  \text{(i)}\ &  \ \text{$f$ is continuous,} \\
  \text{(ii)}\ & \ \text{$f(x)=O((1+\|x\|_2)^{-n-\delta})$ for some $\delta>0$, and} \\
  \text{(iii)}\ & \ \text{the right-hand side of (\ref{PSF}) is absolutely convergent.}
  \endaligned
\end{equation}
See Part~\ref{item:PSF_holds} of Theorem~\ref{thm:PSF} in Appendix~\ref{sec:appendix} for a proof that these conditions are sufficient for~\eqref{PSF} to hold.

The following theorem generalizes (and slightly improves\footnote{Banaszczyk stated a slightly weaker result for the case when $v \neq 0$, but it is clear his proof gives more.})
 the main tail bound in Banaszczyk's seminal work \cite{Banaszczyk}.

\begin{thm}[Generalized tail bounds]\label{mainthm}
	Assume  that a real-valued, positive   function $f$  satisfies conditions \eqref{PSFconditions}, and its Fourier transform $\widehat{f}$   is  real-valued, nonnegative, and  monotonically non-increasing on  rays, i.e.,   $0 \leq \widehat{f}(tv)\le \widehat{f}(v)$ for all $v\in \R^n$ and $t\ge 1$.\footnote{Note that the non-negativity of $\widehat{f}$ implies   $ 0 < f \le f(0)$, and that condition (\ref{PSFconditions})(iii) is equivalent to the convergence of the right-hand side of (\ref{PSF}) at $v=0$.}  Then the following statements hold for any lattice $\L \subset \R^n$.
	\begin{enumerate}
		\item \label{item:main_one} For any $v \in \R^n$ and $t \geq 1$,
				\begin{equation}\label{lem1}
				\sum_{\l\in\L}f(\smallf{\l+v}{t})
				\ \ \le \ \  t^n\sum_{\l\in\L}f(\l)\,.
				\end{equation}
		
		\item \label{item:main_two}
		For any subset $K \subset \R^n$ and any $v \in \R^n$,
		\begin{equation}\label{ineq2b}
		\sum_{\srel{\l\in\L}{\l + v \notin K}} f(\l+v) \ \ \leq \ \ \nu_f(K)  \sum_{\l\in\L}f(\l)\,,
		\end{equation}
		where
		\begin{equation}\label{maximalf}
		\nu_f(K) \ \ := \ \ \inf_{0 \,< \,u\, \le \, 1}\, \sup_{x \, \notin \, K} \f{f(x)}{u^n f(ux)}
\,,
		\end{equation}
 provided the right-hand side is finite.	
		\item\label{item:main_three}
		If no non-zero lattice vectors lie in $K \subset \R^n$, then
		\begin{equation}
		\label{eq:part3}
		\sum_{\l\in\L^*}\widehat{f}(\l+v) \ \ \ge \ \ (1-2\nu_f(K))   \sum_{\l \in \L^*} \widehat{f}(\l)\,,
		\end{equation}
		provided that the left-hand side is convergent and $\nu_f(K)<\infty$.
		
	\end{enumerate}
\end{thm}
\begin{proof}
	Part~\ref{item:main_one} follows immediately from the Poisson summation formula \eqref{PSF}, the assumptions, and a second application of \eqref{PSF}:
	\begin{equation}\label{ineq1}
	\sum_{\l\in\L}f(\smallf{\l+v}{t})
	\ \ \le \ \ \f{t^n}{|\L|}\sum_{\l\in\L^*}\widehat{f}(t\l)  \ \ \le \ \
	\f{t^n}{|\L|}\sum_{\l\in\L^*}\widehat{f}(\l) \ \ = \ \ t^n\sum_{\l\in\L}f(\l)\,,
	\end{equation}
	for any $v\in \R^n$ and $t\ge 1$.
	
	For Part~\ref{item:main_two}, we have for $0<u\le 1$ that
		\begin{equation}
		\label{eq:ineq2a}
\aligned
		\sum_{\l \in \L} f(u (\l + v)) \ \
		&\geq \sum_{\srel{\l \in \L}{\l + v \notin K}} f(u(\l+v))  \\
		&\geq  \ \ \inf_{x \notin K} \frac{f(ux)}{f(x)}  \!\!  \sum_{\srel{\l \in \L}{\l + v \notin K}} f(\l+v)
		\, .
\endaligned
		\end{equation}

	At the same time, we have $\sum_{\l \in \L} f(u (\l + v))  \leq u^{-n} \sum_{\l \in \L} f(\l)$  by Part~\ref{item:main_one}, from which~\eqref{ineq2b} is immediate.
	
	Finally, for Part~\ref{item:main_three}  consider the Poisson summation formula~\eqref{PSF} applied in the case $t=1$ and $v=0$ to the function $f(x)e(-x\cdot w)$ instead of $f(x)$, where $w$ is an arbitrary vector in $\R^n$.  The Fourier transform of this function is $\widehat{f}(x+w)$.  The assumption that the left-hand side of \eqref{eq:part3} converges thus shows that conditions~\eqref{PSFconditions} hold for $f(x)e(-x\cdot w)$, and hence
	\begin{align}
	\sum_{\l\in\L^*}\widehat{f}(\l+w) \ \ &= \ \ |\L|\,\sum_{\l\in\L}f(\l)\,e(-\l\cdot w)
	\nonumber \\
	&\ge \ \ |\L|\,f(0) \ - \ |\L|\sum_{\srel{\l\in\L}{\l \notin K}}f(\l)
	\nonumber \\
	&= \ \ |\L|\, \sum_{\l \in \L} f(\l) \ - \ 2 \,|\L| \sum_{\srel{\l\in\L}{\l \notin K}}f(\l)
	\nonumber \\
	&\geq  \ \ (1-2\nu_f(K))   |\L| \sum_{\l \in \L} f(\l)
	\nonumber \\
	&= \ \  (1-2\nu_f(K))  \sum_{\l \in \L^*} \widehat{f}(\l)
	\; ,
	\end{align}
	as claimed.
\end{proof}

Many functions $f$ of interest (and all of the functions that we consider in the sequel) satisfy an additional concavity property:
\begin{equation}\label{ratiomonot}
		\frac{f(u x)}{f(x)} \ \ \geq\ \  \frac{f(ut x)}{f(t x)}
		\end{equation}
for any $x \in \R^n$ and $u,t \in (0,1]$.
When this is the case and $K$ is sufficiently nice, the supremum in the definition of $\nu_f(K)$ can be replaced by a maximum over the boundary of $K$.
If the function $f$ also factors through the norm function~\eqref{normK} (like the Gaussian factors through the $\ell_2$ norm) then $\nu_f(K)$ takes a particularly nice form, as the following corollary shows. 

\begin{cor}
	\label{cor:main_theorem_norms}
Let  $K\subset \R^n$ be a compact set  whose interior contains the origin and which is starlike with respect to the origin.
Let $g : \R_{\geq 0} \to \R_{> 0}$ be an injective function for which the composition $f(x) = g(\|x\|_K)$ satisfies~\eqref{ratiomonot} and the requirements of Theorem~\ref{mainthm} (i.e., \eqref{PSFconditions} and the  monotonically non-increasing on rays condition).
Then for any $r > 0$,
\begin{equation}
	\label{eq:param_for_norms}
	\nu_f(rK) \ \  \le \  \ \mu_g(r)
	\; ,
\end{equation}
	where
	\begin{equation}\label{maximalf_norms}
	\mu_g(r) \ \ := \ \ \frac{g(r)}{\sup_{0 < u \le 1} u^{n}  g(ur)}
	\,.
	\end{equation}
	In particular, for any lattice $\L \subset \R^n$,
	 \begin{equation}\label{ineq_for_norms}
	 \sum_{\srel{\l\in\L}{\|\lambda + v\|_K \, \geq \, r}} f(\l+v) \ \ \le \ \ \mu_g(r) \, \sum_{\l\in\L}f(\l)
	 \, .
	 \end{equation}
\end{cor}
\begin{proof}
	Since $g$ is injective, we have that $f(y) = g(s)$ if and only if $s=\|y\|_K$.
Thus for any fixed $u>0$ the value of $f(ux)=g(\|ux\|_K)=g(u\|x\|_K)$ depends only on $\|x\|_K$.
This implies that
	\[
	\nu_f(rK) \ \ = \ \ \inf_{0 \,< \,u\, \le \, 1}\, \sup_{x \notin rK}   \f{f(x)}{u^n f(ux)}  \ \ = \ \  \inf_{0 \,< \,u\, \le \, 1}\, \sup_{s > r}   \f{g(s)}{u^n g(us)}
	\,,	\]
where in the last equality we have used the fact $K$ is starlike.
	Finally, by~\eqref{ratiomonot}, we see that for any $s > r$, $g(s)/g(us) \leq g(r)/g(ur)$, so that $\nu_f(rK) \leq \mu_g(r)$.
	The result then follows immediately from Part~\ref{item:main_two} of Theorem~\ref{mainthm}.
	\end{proof}

From Theorem~\ref{mainthm}, we derive the following general transference bound.  Recall the definition of $\nu_f(\cdot)$ from \eqref{eq:nu_def_intro} and the definitions of $\sigma_K(\cdot)$ and $\rho_K(\cdot)$ from \eqref{eq:sigma_K_def}-\eqref{eq:rho_K_def}.

\begin{thm}[Generalized transference bound]
	\label{thm:generic_transference}
	Assume that $f, \widehat{f} > 0$ each satisfy all conditions of Theorem~\ref{mainthm} (i.e., \eqref{PSFconditions} and the  monotonically non-increasing on rays condition). Suppose that $K, K' \subset \R^n$ are compact sets with the origin in their interiors and which are starlike with respect to the origin such that
	\begin{equation}
	\label{eq:chi_bound}
	2\nu_f(K) + \nu_{\widehat{f}}(K') \ \ < \ \ 1
	\, .
	\end{equation}
	Then
	for any lattice $\L \subset \R^n$
	\begin{equation}
	\label{eq:generic_transference}
	\sigma_{K}(\L) \, \rho_{K'}(\L^*) \ \ \leq  \ \ 1
	\,.
	\end{equation}
\end{thm}
\begin{proof}
	It follows from definitions~\eqref{eq:sigma_K_def} and~\eqref{eq:rho_K_def} that the left-hand side of \eqref{eq:generic_transference} is unchanged if $\Lambda$ is replaced by a scaling $t\L$.
	We thus assume, as we may by rescaling,  that $\sigma_{K}(\L) = 1$.  By definition, $s \Lambda$ then has no non-zero vectors in $K$ for any $s>1$.  Part~\ref{item:main_three} of Theorem~\ref{mainthm}, when applied to the lattice $s\Lambda$ (which has dual lattice $(s\Lambda)^*=s\i \Lambda^*$), then shows that
	\begin{equation}
	\sum_{\l\in\L^*}\widehat{f}(s\i(\l+v)) \ \ \ge \ \ (1-2 \nu_f(K) )  \sum_{\l \in \L^*} \widehat{f}(s\i\l)
	\; ,
	\end{equation}
	for any $v \in \R^n$ and any $s > 1$.
	
	By Part~\ref{item:main_two} applied to $\widehat{f}$, $s\i\Lambda^*$, and $K'$,
	\begin{equation}
	\sum_{\stackrel{\l\in\L^*}{s\i(\l + v) \,\notin\, K'}} \widehat{f}(s\i(\l+v)) \ \ \leq \ \ \nu_{\widehat f}(K')  \sum_{\l\in\L^*}\widehat{f}(s\i\l)
	\, .
	\end{equation}
	Since $\nu_{\widehat f}(K') <1-2\nu_f(K)$,
	we have
	\begin{equation}
	\sum_{\stackrel{\l\in\L^*}{s\i(\l + v) \, \notin \, K'}} \widehat{f}(s\i(\l+v))   \ \ < \ \  \sum_{\l\in\L^*}\widehat{f}(s\i(\l+v))
	\end{equation}
	for all $v \in \R^n$.  Hence for any $v\in \R^n$ there must exist some $\l \in \L^*$ such that $\l + v \in sK'$;  that is,
	$\rho_{K'}(\L^*) \leq s$. Since this holds for all $s > 1$, we deduce that $\rho_{K'}(\L^*) \leq 1$, as needed.
\end{proof}

\section{Applications of Theorems~\ref{mainthm} and~\ref{thm:generic_transference}}\label{sec:Applications}

In this section we consider various admissible pairs of functions.  We begin first with some facts about the Fourier transform in  $n=1$ dimension:
\begin{itemize}
  \item if $f(x)=e^{-\pi x^2}$, then $\widehat{f}(x)=f(x)$;
  \item if $f(x)= \sech(\pi x)$, then $\widehat{f}(x) = f(x)$;
  \item if $f(x) = (1+2\cosh(2\pi x/\sqrt{3}))^{-1}$, then $\widehat{f}(x) =  f(x)$;
   \item   if $f(x)=e^{-|x|}$, then $\widehat{f}(x)=\frac{2  }{1+4 \pi ^2 x^2}$; and
  \item if $f(x)=e^{-|x|^p}$ with $0<p\le 2$, then $\widehat{f}\ge 0$ (see \cite[Lemma 5]{EOR91}).
\end{itemize}
In the rest of this section we more generally study functions of the form
\begin{equation}\label{theform}
 f(x_1,\ldots,x_n) \ \ = \ \ \prod_{j=1}^n f(x_j)\,, \ \  \widehat{f}(x_1,\ldots,x_n) \ \ = \ \ \prod_{j=1}^n  \widehat{f}(x_j)\,,
\end{equation}
where each $f$ is one of these examples (one could further consider functions of the form $\prod_{j=1}^n f_j(x_j)$, though we shall not do so here).

\subsection{Recovering Banaszczyk's  bounds \texorpdfstring{\cite{Banaszczyk}}{}}
\label{sec:banaszczyk_recovered}

As our first example, we take $f(x) = \widehat{f}(x) = e^{-\pi \|x\|_2^2}$ to be a Gaussian, as in Banaszczyk's original application. From this, we immediately derive what is essentially Banaszczyk's original transference theorem for the Euclidean norm~\cite[Theorem 2.2]{Banaszczyk}.\footnote{Though Banaszczyk's theorem states that $\sigma_2(\L)\rho_2(\L^*) \leq n/2$, he remarks towards the end of his paper that a more careful analysis yields a bound like~\eqref{eq:banaszczyk_transference_recovered}.  He also proves that there exist   lattices  $\Lambda$ in arbitrarily large dimensions with $\sigma_2(\L)\rho_2(\L^*)\gg n$.  In fact, his $n/2$ bound has the optimal constant $C$ among bounds of the form $Cn$,
 since $\sigma_2(\Z) \rho_2(\Z) = 1/2$.  He also proved additional transference bounds relating successive minima, a topic which we have chosen to omit for the sake of brevity.}

\begin{thm}[$\ell_2$ transference bound]
	\label{thm:banaszczyk_recovered}
	For any $\L \subset \R^n$, let  $\sigma_2(\L) := \min_{\l \in \L_{\neq 0}} \|\l \|_2$ denote the length of the its shortest non-zero vector in the Euclidean norm, and let $\rho_2(\L^*) := \max_{v \in \R^n} \min_{\l \in \L^*} \|\l - v\|_2$ denote the covering radius of its dual lattice in the Euclidean norm. Then
	\begin{equation}
	\label{eq:banaszczyk_transference_recovered}
	\sigma_2(\L)\,\rho_2(\L^*)  \ \ \le \ \ \frac{n}{2\pi} + \frac{3\sqrt{n}}{\pi}
	\, .
	\end{equation}
\end{thm}

\begin{proof}%
		Let $f(x) = \widehat{f}(x) := e^{-\pi \|x\|_2^2}$,
		$\tau := \f 12 + \f {3}{\sqrt{n}}$, $r := \sqrt{\tau n/\pi}$, and $K := \{ x \in \R^n  : \|x\|_2 \leq 1\}$.
		By Corollary~\ref{cor:main_theorem_norms},
		\[
		\nu_{f}(rK) \ \  \leq  \ \ \frac{e^{-\pi r^2}}{\sup_{0 < u \leq 1} u^n e^{-\pi u^2 r^2}} \  \ = \ \  (2e^{1-2\tau} \tau)^{n/2} \ \  = \ \  (1+6/\sqrt{n})^{n/2} e^{-3 \sqrt{n}}
		\; ,
		\]
using the fact that the supremum in the denominator occurs at $u=\frac{\sqrt{n}}{\sqrt{2 \pi } r}$.
		A straightforward computation then shows that $3\nu_f(rK) < 1$.
		Applying Theorem~\ref{thm:generic_transference}, we see that
		$
		\sigma_{rK}(\L) \rho_{rK}(\L^*) \leq 1
		$.
		The result then follows by the scaling formulas $\sigma_2(\L) = r\sigma_{rK}(\L)$ and $\rho_2(\L^*) = r\rho_{rK}(\L^*)$, so that
		$
			\sigma_2(\L) \rho_2(\L^*)  \leq r^2
		$,
		as was to be shown.
\end{proof}

It is interesting to speculate whether or not (\ref{eq:banaszczyk_transference_recovered}) can be improved by using carefully optimized test functions.  Banaszczyk's choice of the Gaussian appears to be particularly natural among functions of the form $f(x)=g(\|x\|_2)$, with $g$ fixed and the dimension $n$ varying. This is because such $f$ which are bounded, continuous, and integrable on $\R^n$, and which furthermore have non-negative Fourier transform $\widehat{f}$, can be expressed using  Schoenberg's theorem  as
\begin{equation}\label{posdef}
  f(x) \ \ = \ \ \int_0^\infty e^{-\pi t^2\|x\|_2^2}\,d\alpha(t)
\end{equation}
for some nonnegative Borel measure $\alpha$ on $(0,\infty)$~\cite{Schoenberg}.  By the Fubini theorem, functions of the form (\ref{posdef}) are integrable on $\R^n$ if and only if $\int_0^\infty t^{-n}d\alpha(t)<\infty$, in which case the Fourier transform
\begin{equation}\label{posdefft}
  \widehat{f}(r) \ \ = \ \ \int_0^\infty e^{-\pi\|r\|_2^2/t^2}\,t^{-n}\,d\alpha(t)
\end{equation}
has a similar form.  Gaussians correspond to when the measure $\alpha$ is concentrated at a single point.  When the measure has larger support, a heuristic argument replacing these integrals by finite sums of Gaussians shows that  the best-possible constants in (\ref{eq:banaszczyk_transference_recovered}) are achieved for a single Gaussian.  This suggests that improving  (\ref{eq:banaszczyk_transference_recovered}) would require functions beyond simply those of the form $f(x)=g(\|x\|_2)$, where $g$ is independent of $n$.

\subsection{A transference bound in the \texorpdfstring{$\ell_1$}{ell\_1} norm}
\label{sec:ell_1_transference}

In this subsection, we take
\begin{equation}\label{gbetadef}
  f(x) = f(x_1,\ldots,x_n)  \ \ := \ \  \prod_{i\,=\,1}^n \frac{1}{1+2 \cosh(2\pi x_i/\sqrt{3})} \; .
\end{equation}
As   noted above, this function possesses  the Fourier duality  $\widehat{f}(x) =f(x)$ in analogy to Gaussians.  However, its asymptotics
 $\log(f(x)) \approx -2\pi \|x\|_1/\sqrt{3}$  are related to   the $\ell_1$ norm (as opposed to the $\ell_2$ norm  for Gaussians).

\begin{lem}
	\label{lem:cosh_thing}
Let \[
	C^*  \ \  :=  \ \   \max_{z \geq 0} \Big( z -  \frac{z \tanh(z)}{1+\sech(z)/2} \Big) \ \  \approx \ \  0.42479\,,
	\]
and let $$K_\alpha \ \  := \ \  \{ x \in \R^n  :  \|x\|_1 \leq (1 + C^* )\alpha n\}$$ be the $\ell_1$ ball of radius $(1 + C^* )\alpha n$.  Then
for any $\alpha > \f{\sqrt{3}}{2\pi }$,
	\begin{equation}
	\label{eq:sech_chi}
	\nu_{f}(K_\alpha) \ \ \leq \ \  \Big( \frac{2\pi \alpha}{\sqrt{3}}\Big)^n   e^{-(\f{2\pi \alpha}{\sqrt{3}} -1 )n}
	\, .
	\end{equation}
\end{lem}
\begin{proof}
	Let $x = (x_1,\ldots, x_n) \in \R^n$.
	By differentiating $\log f(ux)$ with respect to $u$, we see that
	\begin{align*}
	\log(f(u x)) - \log(f(x))
	& \ \ = \ \  \frac{2\pi}{\sqrt{3}}\sum_i  \int_u^1 \frac{x_i \tanh(2\pi v x_i/\sqrt{3})}{1+\sech(2\pi v x_i/\sqrt{3})/2}\, {\rm d} v \\
	& \ \ \geq \ \  \frac{(1-u)2 \pi}{\sqrt{3}} \cdot \sum_i  \frac{x_i \tanh(2\pi u x_i/\sqrt{3})}{1+\sech(2\pi u x/\sqrt{3})/2}
	\, ,
	\end{align*}
	where the inequality follows from the fact that the integrand is monotonically non-decreasing in $v$.
	
	Next, we note that
	\begin{align*}
	|x_i| -   \frac{ x_i\tanh(2\pi u x_i/\sqrt{3})}{1+\sech(2\pi u x_i/\sqrt{3})/2}
	& \ \ \leq  \ \ \frac{\sqrt{3}}{2\pi u} \cdot \max_{z \geq 0} \Big( z -  \frac{z\tanh(z)}{1+\sech(z)/2} \Big)\\
	& \ \  = \ \  \frac{\sqrt{3}}{2\pi u} \cdot C^* \; .
	\end{align*}
	Therefore,
	\[
	\log(f(u x)) - \log(f(x))   \ \ > \ \ \frac{(1-u)2 \pi n}{\sqrt{3}}   \(\frac{\|x\|_1}{n} - \frac{\sqrt{3} C^*}{2\pi u}\)
	 .
	\]
	Taking  $u =\f{\sqrt{3}}{2\pi \alpha} < 1$, it follows that
	\[
	u^n \frac{f(ux)}{f(x)} \ \ > \ \ \Big( \frac{2\pi \alpha}{\sqrt{3}}\Big)^{-n} e^{(\f{2\pi \alpha}{\sqrt{3}} -1) (\f{\|x\|_1}{\alpha n} - C^*)}
	\; ,
	\]
	so that in particular
	\[
		u^n \frac{f(ux)}{f(x)} \ \ > \ \ \Big( \frac{2\pi \alpha}{\sqrt{3}}\Big)^{-n}  e^{(\f{2\pi \alpha}{\sqrt{3}} -1 )n}
	\]
	for $x \notin K_\alpha$ (i.e., $\|x\|_1 > (1 + C^* )\alpha n$). The result now follows after recalling the definition of $\nu_f(\cdot)$ in   (\ref{eq:nu_def_intro}).
\end{proof}

\begin{thm}[$\ell_1$ transference bound]
	\label{thm:ell1_transference}
		For any lattice $\L \subset \R^n$, let  $\sigma_1(\L) := \min_{\l \in \L_{\neq 0}} \|\l \|_1$ denote the length of  its shortest non-zero vector in the $\ell_1$ norm, and let $\rho_1(\L^*) := \max_{v \in \R^n} \min_{\l \in \L^*} \|\l - v\|_1$ denote the covering radius of its dual lattice in the $\ell_1$ norm. Then
		\begin{equation}
		\label{eq:ell1_transference}
		\sigma_1(\L)\,\rho_1(\L^*)  \ \ < \ \ 0.154264 n^2 \cdot \(1 + 2 \pi\sqrt{\smallf 3n} \)^2
		\, .
		\end{equation}
\end{thm}
\begin{proof}
	Let $\alpha :=  \f{\sqrt{3}}{2\pi } + \f{3}{\sqrt{n}}$, and set $K_\alpha := \{ x \in \R^n  :  \|x\|_1 \leq (1 + C^*) \alpha n\}$  and $C^* = 0.42479\cdots$ as in the statement of Lemma~\ref{lem:cosh_thing}. Applying the lemma, we have
	\[
	\nu_{f}(K_\alpha)  \ \  \le  \ \   \Big( \frac{2\pi \alpha}{\sqrt{3}}\Big)^n \cdot e^{-(\f{2\pi \alpha}{\sqrt{3}} -1 )n}  \ \ < \ \   \frac{1}{3}
	\; ,
	\]
	 where the second inequality follows by a straightforward computation.
	Therefore, $3\nu_{f}(K_\alpha) < 1$.  It is straightforward to verify that $f=\widehat{f}$ obeys the assumptions of Theorem~\ref{thm:generic_transference}, and hence
	\[
	\sigma_{K_\alpha}(\L) \rho_{K_\alpha}(L)  \ \ \leq \ \  1
	\; .
	\]
	We then obtain the result by simply noting that
	$
	\sigma_1(\L) =  (1 + C^*) \alpha n \cdot \sigma_{K_\alpha}(\L)
	\; ,
	$
	and similarly $\rho_1(\L^*) =  (1 + C^*) \alpha n \cdot \sigma_{K_\alpha}(\L) $, so that their product is at most
	\[
	(1 + C^*)^2 \alpha^2 n^2 \ \  < \ \  0.154264 n^2 \cdot (1 + 2 \pi\sqrt{3/n} )^2
	\; ,
	\]
	as needed.
\end{proof}

\subsection{Supergaussians, \texorpdfstring{$\ell_p$}{ell\_p} norms for \texorpdfstring{$0 < p\le 2$}{0<p<=2}, and the kissing number}
\label{sec:ell_p_kissing}

Here, we consider the following specialization of Theorem~\ref{mainthm} to functions of the form $f(x) := \exp(-\|x\|_p^p)= e^{-(|x_1|^p + \cdots + |x_n|^p)}$, which are sometimes referred to as ``supergaussians.''
\begin{lem}
	\label{lem:ell_p_gaussians}
	Let $0 < p \leq 2$ and $f_p(x) := \exp(-\|x\|_p^p)$. Then
	\[
	\sum_{\srel{\l\,\in\,\L}{\|\lambda + v\|_p \, \geq\, t(n/p)^{1/p}}} f_p(\l+v) \ \ \le \ \ \big( e t^p  e^{-t^p} \big)^{n/p} \, \sum_{\l\,\in\,\L}f_p(\l)
	\]
for any $t \geq 1$.
\end{lem}

\begin{proof}
	We apply Corollary~\ref{cor:main_theorem_norms} to $f=f_p$.
It is well-known (see, for example, \cite[Lemma 5]{EOR91}) that
the single-variable function
  $x\mapsto e^{-|x|^p}$ has the form (\ref{posdef}).  Since it is integrable, its Fourier transform has the form (\ref{posdefft}) with $n=1$, and   is in particular non-negative and non-increasing on rays.  Furthermore, a straightforward computation shows that $f_p$ satisfies \eqref{ratiomonot}.
  The only remaining condition to show is (\ref{PSFconditions})(iii), which is the absolute convergence of the right-hand side of the Poisson summation formula.  This follows from the fact that the Fourier transform $\int_\R e^{-|x|^p}e^{-2\pi ir x}{\rm d}x $ of  $e^{-|x|^p}$  is asymptotic to
  $-\frac{\pi ^{-p-\frac{1}{2}} |r|^{-p-1} \Gamma \left(\frac{p+1}{2}\right)}{\Gamma \left(-\frac{p}{2}\right)}$ for $0<p<2$ (see, for example, \cite{Sidi} for a
  recent treatment of the asymptotics of Fourier integrals with singularities).
It follows that for   $r := t(n/p)^{1/p}$,
		\[
		\sum_{\srel{\l\in\L}{\|\lambda + v\|_p \geq r}} f_p(\l+v) \ \ \le \ \ \mu_p(r) \sum_{\l\in\L}f_p(\l)
		\]
		with
		\[
		\mu_p(r) \ \ := \ \ \frac{e^{-r^p}}{\sup_{0 < u \le 1} u^{n}  e^{-(u r)^p}}
		\,.
		\]
		A simple computation shows that $\mu_p(r) = (epr^p/n)^{n/p}e^{-r^p}$.
\end{proof}

From this, we derive an upper bound of $e^{n/p + o(n/p)}$ on the lattice kissing number or lattice Hadwiger number --- the number of non-zero lattice points with minimal length --- in $\ell_p$ norms for $0 < p \leq 2$.
To the authors' knowledge, the only previously known bounds on these quantities for $p \neq 2$ were the trivial bounds $2(2^n - 1)$ for $1 < p < 2$ and $3^n - 1$ for $p = 1$. (Much better bounds are known for $p = 2$ using sophisticated techniques~\cite{KL78}, and as far as we know nothing was known for $p < 1$.) Talata also provided evidence for a conjectured upper bound of $1.5^{n + o(n)}$ for the $ p = 1$ case. See~\cite{Swan17} for a recent survey of such results.
We actually prove a slightly more general bound of $e^{u^p n/p + o(u^p n/p)}$ on the ``$u$-handshake number'' number, which is the number of non-zero lattice points whose length is within a factor $u \geq 1$ of the minimal length.\footnote{We note that this quantity must be unbounded as $p \to 0$, as even in $n=2$ dimensions there exist lattices with infinitely many non-zero lattice points $\lambda = (\lambda_1,\ldots, \lambda_n)$ such that $\prod_i |\lambda_i|$ is minimal. (For example, take the canonical embedding of the ring the integers of a number field having infinitely many units.) Since $\|\l\|_p^p \sim n+ p \sum_i\log |\lambda_i|$ as $p \to 0$, this implies that the $u$-handshake number for such lattices and $u > 1$ is unbounded as $p \to 0$.}  Thus the kissing number is simply the $1$-handshake number.

\begin{thm}[$\ell_p$ handshake number bound]
	\label{thm:handshake}
	For any $0 < p \leq 2$ and lattice $\L \subset \R^n$, let
	$\sigma_p(\L) := \min_{\l \in \L_{\neq 0}} \|\l\|_p$. Then
	\begin{equation}
	\label{eq:handshake}
	\#\{\l \in \L_{\neq 0} \ : \ \|\l\|_p \ \leq \  u\, \sigma_p(\L) \}   \ \  \leq  \ \ 10\,\smallf{e^{u^p} n}{p}  \, e^{u^p n/p}
	\;
	\end{equation}
for any $u \ge 1$.
	In   particular, when $u=1$  this shows  that  the lattice kissing number in the $\ell_p$ norm is   $O(\f np e^{n/p})$ for all $0 < p \leq 2$.
\end{thm}
\begin{proof}
	Let $f_p(x) := \exp(-\|x\|_p^p)$. By scaling the lattice appropriately, we may assume that $\sigma_p(\L) = t (\frac np)^{1/p}$ for $t := (1+\frac pn)^{1/p}$. The Theorem shows
		\[
		\sum_{\l \in \L_{\neq 0}} f_p(\l)  \ \ =  \ \ \sum_{\srel{\l \in \L}{\|\l\|_p \,\geq\, \sigma_p(\L)}} f_p(\l) \ \  \leq  \ \ \big( e t^p  e^{-t^p} \big)^{n/p} \, \sum_{\l\in\L}f_p(\l)
		\, .
		\]
		Noting that $\sum_{\l\in\L}f_p(\l) = 1 + \sum_{\l \in \L_{\neq 0}} f_p(\l)$ and rearranging, we see that
		\[
		\sum_{\l \in \L_{\neq 0}} f_p(\l) \ \  \leq \ \  \frac{\big( e t^p  e^{-t^p} \big)^{n/p}}{1-\big( e t^p  e^{-t^p} \big)^{n/p}} \ \   \leq \ \   \smallf{10n}{p}\,\big( e t^p  e^{-t^p} \big)^{n/p}
			\, ,
		\]
		where in the last inequality we have used the  fact $1-\f{(1+x\i)^x}{e} \geq \f{1}{10x}$ for $x=\f np \geq \f 12$. 
		Let $S$ denote the set of $\l \in \L_{\neq 0}$ with $\|\l\|_p \leq u \sigma_p(\L) = ut (\f np)^{1/p}$. Then
		\[
		\sum_{\l \, \in\, \L_{\neq 0}} f_p(\l)  \ \ \geq \ \  \sum_{\l \, \in\, S} f_p(\l) \ \  \geq \ \   e^{-u^p t^p n/p}\,|S|
		\; .
		\]
		Combining the two inequalities, rearranging, and then using the fact that $(1+\f pn)^{n/p}\le e$, we obtain
		\[
		|S|  \ \ \leq  \ \ \smallf{10n}{p}\,\big( e t^p  e^{(u^p-1)t^p} \big)^{n/p} \ \ \leq  \ \  \smallf{10n}{p}\, e^{(1+n/p)u^p}
		\; ,
		\]
		as was to be shown.
\end{proof}

\appendix

\section{The Poisson Summation Formula}\label{sec:appendix}

Here, we state and prove a version of the Poisson summation formula flexible enough for our applications.  The notation $\|x\|=\|x\|_2$ refers to the $\ell_2$ norm.
\begin{thm}\label{thm:PSF}
Let $f(x)$ denote   a continuous, complex-valued function on $\R^n$ which is $O((1+\|x\|)^{-n-\delta})$ for some $\delta>0$.

\begin{enumerate}
	\item \label{item:fourier_inv} The Fourier inversion formula
\begin{equation}\label{appfourinv_thm}
  f(x) \ \ = \ \ \int_{\R^n}\widehat{f}(r)\,e(r\cdot x)\,dr
\end{equation}
holds provided $f$'s Fourier transform $\widehat{f}(x)=\int_{\R^n}f(r)e(-r\cdot x)dr$ is integrable (i.e., $\int_{\R^n}|\widehat{f}(x)|dx <\infty$).

\item \label{item:PSF_holds} The Poisson summation formula
\begin{equation}\label{appPSF}
  \sum_{\l\in\L}f(\smallf{\l+v}{t}) \ \ = \ \ \f{t^n}{|\L|}\sum_{\l\in\L^*}\widehat{f}(t\l)\,e(t\l\cdot v)\,, \ \ \ t\,>\,0 \ \ \text{and} \ \ v\,\in\,\R^n\,,
\end{equation}
holds provided the right-hand side converges absolutely.
\end{enumerate}
\end{thm}
Both parts of the Theorem are well-known and classical if $f$ is a Schwartz function, or even if both $f(x)$ and $\widehat{f}(x)$ merely satisfy the $O((1+\|x\|)^{-n-\delta})$ bound for some $\delta>0$ (see, for example, \cite[Theorem 2.1]{Cohn}).  Thus the main point here is to relax the condition on the decay of $\widehat{f}$, which is needed in Section~\ref{sec:ell_p_kissing}.
\begin{proof}
Let $\phi\ge 0$ denote a fixed, smooth function supported in the unit ball of $\R^n$ and having total integral $ \int_{\R^n}\phi(x)dx=1$.  For any $0<\e<1$ define the rescaled function $\phi_\e(x)=\e^{-n}\phi(x/\e)$, which also has total integral 1.  We have the estimate
\begin{equation}\label{appposbound}
 |\widehat{\phi}(r)| \ \ \le \ \ \int_{\R^n}\phi(x)\,dx \ \ = \ \ 1 \ \ = \ \ \widehat{\phi}(0)
\end{equation}
 by the non-negativity of $\phi$.

The convolution
\begin{equation}\label{appconv}
  f_\e(x) \ \ := \ \  \int_{\R^n}f(y)\,\phi_\e(x-y)\,dy
\end{equation}
is smooth.  Since $\int_{\R^n}\phi(x)dx=1$,
\begin{equation}\label{apppointwiseprecursor}
 f_\e(x)-f(x) \ \  = \ \ \int_{\R^n}(f(y)-f(x))\,\phi_\e(x-y)\,dy \ \ \le \ \
 \max_{y\in B_\e(x)}|f(y)-f(x)|
\end{equation}
where $B_{\epsilon}(x)$ denotes the closed $\ell_2$ ball of radius $\epsilon$ around $x$.
Therefore
 \begin{equation}\label{apppointwise}
   \lim_{\e\rightarrow 0}\ f_\e(x) \ \ = \ \ f(x)
 \end{equation}
   by the assumed continuity of $f$.

We may bound $f_\e(x)$ using the compact support of $\phi$ as
\begin{equation}\label{appdecayf}
  |f_\e(x)| \  \ll  \ \e^{-n}\int_{\R^n}(1+\|y\|)^{-n-\delta}\,\phi\(\frac{x-y}\e\)dy
  \  \ll \ \e^{-n}\int_{B_{\epsilon}(x)}(1+\|y\|)^{-n-\delta}\,dy\,.
\end{equation}
  The boundedness of the integrand shows that this is $O(1)$.  For $\|x\| \ge 2$ and $y\in B_{\e}(x)$, we have $\|y\| \ge \|x\|-\e \ge \f{1}{2}\|x\|$, and thus the right-hand side of~\eqref{appconv} is $O(\|x\|^{-n-\delta})$.  Combining these two estimates, we see that
\begin{equation}\label{appfbound}
 f_\epsilon(x) \ \ = \ \  O((1+\|x\|)^{-n-\delta})\,,
\end{equation}
independently of $\epsilon$ ---   the same bound that we assumed $f(x)$ satisfies.

  In particular the Fourier transform of $f_\epsilon$ is well-defined, and a change of variables shows it factors as
  \begin{equation}\label{appfactorFT}
    \widehat{f}_\e(x) \ \ = \ \ \widehat{f}(x)\,\widehat{\phi}_\e(x) \ \  = \ \ \widehat{f}(x)\,\widehat{\phi}(\epsilon x)\,.
\end{equation}
The decay assumption on $f$ implies that it is integrable, so that $\widehat{f}(x)$ is bounded.
    Since $\phi$ and all its derivatives have compact support,  the Riemann-Lebesgue Lemma implies that $\widehat{\phi}(x)$ decays faster than the reciprocal of any polynomial as $\|x\|\rightarrow \infty$.  It follows that
    \begin{equation}\label{appfeftadmiss}
    \widehat{f}_\epsilon(x) \ \ = \ \  O_\e((1+\|x\|)^{-n-\delta})\,,
    \end{equation}
    where the last subscript indicates that the implied constant depends on $\epsilon$.
 The Fourier inversion formula
\begin{equation}\label{appfourinv}
  f_\e(x) \ \ = \ \ \int_{\R^n}\widehat{f}_\e(r)\,e(r\cdot x)\,dr \ \ = \ \ \int_{\R^n}\widehat{f}(r)\,\widehat{\phi}(\epsilon r)\,e(r\cdot x)\,dr
\end{equation}
is therefore valid for $f_\e(x)$.

If $\widehat{f}(r)$  is integrable, then the bound  $\widehat{\phi}(\e r)\le 1$ from~\eqref{appposbound} and dominated convergence imply that the right-hand side of~\eqref{appfourinv} converges to $\int_{\R^n}\widehat{f}(r)e(r\cdot x)dr$ in the limit as $\epsilon\rightarrow 0$.   Combined with~\eqref{apppointwise}, this proves~\eqref{appfourinv_thm} and hence Part~\ref{item:fourier_inv}.

To finish, we consider Part~\ref{item:PSF_holds}.
Both $f_\epsilon(x)$ and $\widehat{f}_\epsilon(x)$   satisfy the admissibility bound $O((1+\|x\|)^{-n-\delta})$ by~\eqref{appfbound} and~\eqref{appfeftadmiss}.  Therefore the Poisson summation formula~\eqref{appPSF} is valid with $f$ replaced by $f_\epsilon$ (\cite[Theorem~2.1]{Cohn}):
\begin{equation}\label{appPSFforfeps}
  \sum_{\l\in\L}f_\e(\smallf{\l+v}{t}) \ \ = \ \ \f{t^n}{|\L|}\sum_{\l\in\L^*}\widehat{f}(t\l)\,\widehat{\phi}(\epsilon t \lambda)\,e(t\l\cdot v)\,, \ \ \ t\,>\,0 \ \ \text{and} \ \ v\,\in\,\R^n\,,
\end{equation}
where we have used the factorization~\eqref{appfactorFT}.  We now again use~\eqref{appposbound} and dominated convergence to show that the right-hand side of~\eqref{appPSFforfeps} converges to the right-hand side of~\eqref{appPSF}) as $\epsilon\rightarrow 0$, using the assumed absolute convergence of the latter.  To conclude, we apply dominated convergence to the left-hand side (using the bound~\eqref{appfbound} and the pointwise limit~\eqref{apppointwise}) to show that the left-hand side converges $\sum_{\l\in\L}f(\frac{\l+v}{t})$, as was to be shown.
\end{proof}

\end{document}